\documentclass[10pt]{amsart}
\usepackage{amsmath}
\usepackage[usenames,dvipsnames]{color}
\usepackage{parskip}
\usepackage{amsfonts}
\usepackage{amscd}
\usepackage[centertags]{amsmath}
\usepackage{amssymb}
\usepackage[all,cmtip]{xy}
\usepackage[english]{babel}
\usepackage{tabularx}
\usepackage{mathtools}
\usepackage{amsxtra}
\usepackage{euscript}
\usepackage[T1]{fontenc}
\usepackage{doc, exscale, fontenc, latexsym, syntonly}
\usepackage{amsfonts}
\usepackage{amsthm}
\usepackage{graphicx}
\usepackage{xcolor}
\usepackage{tikz-cd}

\newtheorem{theorem}{Theorem}[section]
\newtheorem{corollary}[theorem]{Corollary}
\newtheorem{lemma}[theorem]{Lemma}
\newtheorem{proposition}[theorem]{Proposition}
\newtheorem{remark}[theorem]{Remark}
\theoremstyle{definition}
\newtheorem{definition}[theorem]{Definition}

\numberwithin{figure}{section}
\numberwithin{table}{section}
\newcommand{\cat}{\ensuremath{\mathrm{cat}}}

\newcommand{\id}{\ensuremath{\mathrm{Id}}}
\newcommand{\TC}{\ensuremath{\mathrm{TC}}}
\newcommand{\D}{\ensuremath{\mathrm{D}}}
\newcommand{\pr}{\ensuremath{\mathrm{pr}}}

\begin{document}

\title{Topological Spaces Induced By Homotopic Distance}

\author{Tane Vergili and Ayse Borat}

\date{\today}

\address{\textsc{Tane Vergili}
Karadeniz Technical University\\
Faculty of Science\\
Department of Mathematics\\
Trabzon, Turkey}
\email{tane.vergili@ktu.edu.tr} 

\address{\textsc{Ayse Borat}
Bursa Technical University\\
Faculty of Engineering and Natural Sciences\\
Department of Mathematics\\
Bursa, Turkey}
\email{ayse.borat@btu.edu.tr}



\begin{abstract} Homotopic distance $\D$ as introduced in \cite{MVML} can be realized as a pseudometric on $\mathrm{Map}(X,Y)$. In this paper, we study the topology induced by the pseudometric $\D$. In particular, we consider the space $\mathrm{Map}(S^1,S^1)$ and show that homotopic distance between any two maps in this space is 1. Moreover, while a general proof of the non-compactness of the space $\mathrm{Map}(X,Y)$ is still an open problem, it can be shown that $\mathrm{Map}(S^1,S^1)$ is not compact. 

\end{abstract}

\maketitle

\section{Introduction}

Homotopic distance which is introduced by Macias-Virgos and Mosquera-Lois, is a generalization of Lusternik Schnirelmann category ($\cat$) and topological complexity ($\TC$). One of the importance of this new concept is to give easier proofs of the $\cat$- and $\TC$-related theorems due to the theorems in \cite{MVML} which tells how the homotopic distance behaves under the composition. This leads us to think that new theorems on $\cat$ and $\TC$ can be proved with a help of homotopic distance. 

In this section, we first give a brief background recalling the definitions of Lusternik Schnirelmann category, topological complexity, and homotopic distance and giving the relations between these concepts. Secondly, we talk about three properties of homotopic distance which allow us to understand the distance as an extended pseudometric. Later we introduce the topology induced by this pseudometric. 

In Section~\ref{Section:circle}, we consider a specific space $\mathrm{Map}(S^1, S^1)$. One of the results of this section is that the homotopic distance between any two maps from $S^1$ to $S^1$ is 1. In Section~\ref{section:last}, we consider the general space $\mathrm{Map}(X,Y)$ and introduce some of its topological properties.

\begin{definition} \cite{CLOT} 
	Lusternik Schnirelmann category of a space $X$, $\cat(X)$, is the least non-negative integer $k\geq 0$ if there exists an open covering $\{U_0, U_1, \ldots, U_k\}$ of $X$ such that the inclusion on each $U_i$ is null-homotopic for $i=0,1,\ldots, k$. 

If there is no such a covering, $\cat(X)=\infty$. 
\end{definition}

\begin{definition}\cite{F2}
    Let $\pi: PX\rightarrow X\times X$, by $\pi(\gamma)=(\gamma(0),\gamma(1))$ be the path fibration. Topological complexity of a space $X$, $\TC(X)$, is the least non-negative integer $k\geq 0$ if there exists an open covering $\{U_0, U_1, \ldots, U_k\}$ of $X\times X$ such that there exists a continuous section $s_i:U_i\rightarrow PX$ for each $i=0,1,\ldots, k$.
    
If there is no such a covering, $\TC(X)=\infty$. 
\end{definition}

\begin{definition}\cite{MVML} Let $f,g:X\rightarrow Y$ be continuous maps, the homotopic distance between $f$ and $g$, denoted by $\D(f,g)$, is the least non-negative integer $k$ such that there exists $U_0, U_1, \cdots, U_k$ open subspaces of $X$ covering $X$ satisfying $f\big|_{U_i}\simeq g\big|_{U_i}$ for all $i=0, 1, \cdots, k$. 

If there is no such a covering, $\D(f,g)=\infty$. 
\end{definition}

The relation between homotopic distance $\D$, $\cat$ and $\TC$ can be given as follows. $\D(\id,c)=\cat(X)$ where $c:X\rightarrow X$ is a constant map, and $\D(i_1,i_2)=\cat(X)$ provided that $i_j: X\hookrightarrow X\times X$ for $j=1,2$, given by $i_1(x)=(x,x_0)$ and $i_2(x)=(x_0,x)$. $\D(\pr_1,\pr_2)=\TC(X)$ provided that $\pr_j: X\times X \rightarrow X$ is the projection to the $j$-th factor for $j=1,2$. For proofs and more details, see \cite{MVML}. 

The properties of the homotopic distance which allow us to build a metric space are listed below. 

\begin{proposition} \label{prop:sym} \cite{MVML} If $f,g:X\rightarrow Y$ are maps, then $\D(f, g)= \D(g, f)$.
\end{proposition}

\begin{proposition}  \cite{MVML} If $f,g:X\rightarrow Y$ are maps, then $\D(f, g)=0$ iff $f \simeq g$.
\end{proposition}

\begin{proposition}\cite{MVML} If $f, g, h: X\rightarrow Y$ are maps and $X$ is a normal space, then 
\[
\D(f,h)\leq \D(f,g)+\D(g,h)
\]
\end{proposition}

Suppose $M$ is a set. A  \emph{pseudometric}  $d$ on $M$ is a function
\begin{displaymath}
d: M\times M \to [0,\infty)
\end{displaymath}
that satisfies
\begin{itemize}
	\item[M1)] $d(m,m)=0$
%	\item[M2)] $d(m,n)=0$  if $m=n$
	\item[M2)] $d(m,n)=d(n,m)$, and
	\item[M3)] $d(m,n)\leq d(m,k)+d(k,n)$
\end{itemize}

for all $m,n,k\in M$.  Further $d$ is called a \emph{semimetric} if it satifies all but $\operatorname{M3}$ with the additional condition that
$d(m,n)=0$ implies $m=n$. An \emph{extended pseudometric} on $M$ is a map $d: M\times M \to [0,\infty]$ satisfying the same three axioms.

Let $\mathrm{Map}(X,Y)$ be the set of continuous maps between $X$ and $Y$
\[ \mathrm{Map}(X,Y) = \{ f:X \to Y  \ | \ f  \ \text{is continuous}  \}. \]
Consider the following function 
\begin{align*}
D: \mathrm{Map}(X,Y)  \times \mathrm{Map}(X,Y)  &\to [0,\infty] \\
(f,g) & \mapsto D(f,g)
\end{align*}
Notice that $D$ is extended semimetric on the quotient $\mathrm{Map}(X,Y) / \mathcal{R}$ where $\mathcal{R}$ is the equivalence relation on $\mathrm{Map}(X,Y)$ defined  by
\begin{displaymath}
f  \ \mathcal{R} \ g \ \ \text{iff} \ \ D(f,g)=0.
\end{displaymath} 
We restrict ourselves for $X$ to be a normal space so that $D$ turns into an extended pseudometric on $\mathrm{Map}(X,Y)$. An (extended) pseudometric also  induces an (extended) pseudometric space which is generated by the set of open balls. Then the topology on $\mathrm{Map}(X,Y)$ induced by $D$ is generated by the open balls
\[B_r(f):=\{ g\in \mathrm{Map}(X,Y) \ | \ D(f,g)<r \}  \]
for $r>0$. Observe that $B_r(f)$ consists of maps which are homotopic to $f$ for $r \leq 1$.

\begin{remark} For a simplicial complex $K$, the geometric realization $||K||$ is a normal Hausdorff space (see Theorem 17 in \cite{S}). So, for simplicial complexes $K$ and $L$, we can consider the topology on $\mathrm{Map}(||K||, ||L||)$ which is induced by $\D$. More generally, for a simplicial complex $K$ and a topological space $X$, we can consider the topology on $\mathrm{Map}(||K||, X)$ induced by $\D$.
\end{remark}

Throughout this paper, the ``domain space'' $X$ of $\mathrm{Map}(X,Y)$ is assumed to be normal space. 

\begin{proposition}\label{indiscreteball}
   Let $f\in \mathrm{Map}(X,Y)$. Then the open ball $B_r(f)$ has the indiscrete topology for $r\leq 1$. 
\end{proposition}
\begin{proof}
	We have $B_\varepsilon(f) \supseteq B_r(f)$ for all $f\in B_r(f)$ and $\varepsilon >0$.
\end{proof}

Proposition~\ref{indiscreteball} yields that the open ball $B_r(f)$ is connected for $r\leq 1$. However this is not true for the case $r> 1$ under a certain condition.

\begin{theorem}\label{connected} $B_r(f)$ is not connected for $r>1$, provided $B_1(f)$ is a proper subset of  $B_r(f)$.

\end{theorem}

\begin{proof} Take $U=B_1(f)$. For $r>1$, $U=B_1(f)\subseteq B_r(f)$. $U$ is obviously open. Further, $U$ is also closed, since its closure $\overline{U}=U\cup \{h\in B_r(f) \: | \: D(h,U)=0 \}$ is equal to $U$ where $D(h,U)=\min_{g\in U}D(g,h)$. 

Now take $V=U^c$ which is an open set. Since $B_1(f)$ is a proper subset of  $B_r(f)$, $V$ is non-empty. Hence $U$ and $V$ separates $B_r(f)$. This concludes that $B_r(f)$ is not connected.
\end{proof}

Throughout this paper, when we say ``a space'' we mean the topological space induced by the pseudometric $\D$.

\section{$\mathrm{Map}(S^1,S^1)$}\label{Section:circle}

Consider the space $\mathrm{Map}(S^1,S^1)$ where $S^1$ is the unit circle and let $f \in \mathrm{Map}(S^1,S^1)$. Then the induced map $f_*$ from the fundamental group of the circle $\pi_1(S^1)$ to itself is a group homomorphism $f_*: \pi_1(S^1) \to \pi_1(S^1)$. Note that $\mathrm{Im}(f)$ is a subgroup of $\mathbb{Z}$ so that it is of form $n\mathbb{Z}$ for some $n\in \mathbb{Z}$. This gives us that either $f$ is a constant map or is of form $z\mapsto z^n$ for $n\in \mathbb{Z}$.

\begin{theorem} \label{thm:circle} Consider the space $\mathrm{Map}(S^1,S^1)$. Let $f_n, f_m \in \mathrm{Map}(S^1,S^1)$ of degree $n$ and $m$, respectively. Then $\D(f_n,f_m)=1$.
\end{theorem}
\begin{proof}
	We know that $D(f_n,f_m)\leq \mathrm{cat}(S^1)$ by Corollary 3.9 in \cite{MVML}. The fact that $\cat(S^1)=1$ implies $D(f_n,f_m)\leq 1$. Since the degrees of $f_n$ and $f_m$ are not equal, these maps cannot be homotopic. So $D(f_n,f_m)\neq 0$. Hence $D(f_n,f_m)=1$.
\end{proof}

\begin{corollary}
Consider the space $\mathrm{Map}(S^1,S^1)$ and the map $f_n$ as described in Theorem~\ref{thm:circle} where $n \in 
\mathbb{Z}^+$. Then $D(f_n,c)=1$ where $c$ is any constant map $c: S^1 \to S^1$.
\end{corollary}
\begin{proof}
    The constant map $c$ and $f_n$ are not homotopic and the proof follows similarly from Theorem~\ref{thm:circle}.
\end{proof}%

For $f_n \in \mathrm{Map}(S^1,S^1)$ as described in Theorem~\ref{thm:circle}, observe that $B_r(f_n)=\{f_n\}$ for $r\leq 1$ and $B_r(f_n)=\mathrm{Map}(S^1,S^1)$ for $r>1$. Also 
$B_r(c)=\{\text{all \ constant \ maps \ on} \ S^1\}$ for $0<r\leq 1$ and $B_r(c)=\mathrm{Map}(S^1,S^1)$ for $r> 1$.

\begin{corollary}
	The space $\mathrm{Map}(S^1,S^1)$ is second countable (hence separable and Lindelöf).
\end{corollary}
\begin{proof}
	A countable basis for $\mathrm{Map}(S^1,S^1)$  is $\mathcal{B}=\{ \{f_n\}  : \ n \in \mathbb{Z} \} \cup \{ B_{\frac{1}{2}}(c)\} \}$. In a (extended) pseudometric space being second countable is equivalent with being separable and Lindelof by Lemma~17 in \cite{BubVer:2018}. 
\end{proof}

\begin{corollary}
	The space $\mathrm{Map}(S^1,S^1)$ is not compact.
\end{corollary}
\begin{proof}
	Let $f_n$ be the maps described in Theorem~\ref{thm:circle} and  $c$ be any constant map. Then the open cover $\mathcal{G}=\{ \{f_n\}  : \ n \in \mathbb{Z} \} \cup  \{ B_{\frac{1}{2}}(c)\} \}$ for $\mathrm{Map}(S^1,S^1)$ does not have a finite subcover. 
\end{proof}

\section{Topological properties of $\mathrm{Map}(X,Y)$}\label{section:last}

\begin{lemma}
	Suppose $X$ is an infinite discrete space. Then $D(\id_X,c)=\infty$ where $1_X$ and $c$ are the identity map and constant map, respectively.
\end{lemma}
\begin{proof}
	Any discrete space is normal so that $D$ is a pseudometric on $\mathrm{Map}(X,X)$. Since $X$ is discrete, $\id_X$ and $c$ cannot be homotopic. Hence $D(\id_X,c)>0$. Suppose $D(\id_X,c)=n$ where $n$ is a positive integer. By the definition of the homotopic distance, there exists an open cover $\mathcal{U}=\{U_0,U_1,\dotsc,U_n\}$ for $X$ such that $\id_X \big|_{U_i}\simeq c\big|_{U_i}$ for $i=0,1,\dotsc,n$. $\id_X \big|_{U_i}\simeq c\big|_{U_i}$  implies that $U_i$ is contractible in $X$ so that it is path connected. Since the only path connected subsets of a discrete space are  singletons, $\mathcal{U}$ cannot be a cover for $X$. Thus $D(\id_X,c)=\infty$.
\end{proof}

The space $\mathrm{Map}(X,Y)$ is not interesting whenever $X$ or $Y$ is contractible. 

\begin{theorem}
	If $X$ or $Y$ is contractible, then $\mathrm{Map}(X,Y)$ is indiscrete.  
\end{theorem}
\begin{proof}
	The proof follows from the fact that any two continuous maps from a contractible space to any space or vice versa are homotopic.
\end{proof}

\begin{proposition}
	If $X$ is an infinite discrete space, then $\mathrm{Map}(X,X)$ is not path connected. 
\end{proposition}	
\begin{proof}
	See Lemma 13 in \cite{BubVer:2018}.
\end{proof}

\begin{theorem}  $\mathrm{Map}(X,Y)$ is not connected, provided that $\mathrm{Map}(X,Y)$ is not indiscrete. 
\end{theorem}

\begin{proof} We can choose $U=B_1(f)$ for a fixed $f\in \mathrm{Map}(X,Y)$. Notice that $U$ is non-empty, open and closed (see proof of Theorem~\ref{connected}). 

Define a set $V=\mathrm{Map}(X,Y)\setminus B_1(f)$. More precisely, there is a $g\in \mathrm{Map}(X,Y)$ which is not in $B_1(f)$. If we cannot find such a $g$, then we have $\mathrm{Map}(X,Y)=B_1(f)$ which contradicts with the fact that $\mathrm{Map}(X,Y)$ is not indiscrete. Hence $V$ is non-empty and $\mathrm{Map}(X,Y)$ is not connected.

\end{proof}

\end{document}